\documentclass[10pt,draftcls,onecolumn, journal]{IEEEtran}
\IEEEoverridecommandlockouts                              	
\usepackage{amsmath}
\usepackage{amsfonts}
\usepackage{amssymb}
\usepackage{cite}
\usepackage{hhline}
\usepackage{multirow}
\usepackage{algorithmic}
\usepackage[ruled]{algorithm}
\usepackage{xcolor}
\usepackage{bbm}
\usepackage{lipsum}
\usepackage{graphicx}
\usepackage{epstopdf}
\usepackage{lipsum}
\usepackage{blindtext}
\graphicspath{{./figures/}}
\DeclareGraphicsExtensions{.png}

\newtheorem{remark}{\bfseries Remark}
\newtheorem{definition}{\bfseries Definition}
\newtheorem{theorem}{\bfseries Theorem}
\newtheorem{lemma}{\bfseries Lemma}
\newtheorem{assumption}{\bfseries Assumption}

\begin{document}

\title{An Asynchronous Mini-Batch Algorithm for \\ Regularized Stochastic Optimization}

\author{Hamid Reza Feyzmahdavian, Arda Aytekin, and Mikael Johansson
\thanks{H. R. Feyzmahdavian, A. Aytekin, and M. Johansson are with the Department of Automatic Control, School of Electrical Engineering and ACCESS Linnaeus Center, Royal Institute of Technology (KTH), SE-100 44 Stockholm, Sweden.
Emails: {\tt \{hamidrez, aytekin, mikaelj\}@kth.se.}}
}
\maketitle

%
%

\begin{abstract}
Mini-batch optimization has proven to be a powerful paradigm for large-scale learning. However, the state of the art parallel mini-batch algorithms assume synchronous operation or cyclic update orders. When worker nodes are heterogeneous (due to different computational capabilities or different communication delays), synchronous and cyclic operations are inefficient since they will leave workers idle waiting for the slower nodes to complete their computations. In this paper, we propose an asynchronous mini-batch algorithm for regularized stochastic optimization problems with smooth loss functions that eliminates idle waiting and allows workers to run at their maximal update rates. We show that by suitably choosing the step-size values, the algorithm achieves a rate of the order $\mathcal{O}(1/\sqrt{T})$ for general convex regularization functions, and the rate $\mathcal{O}(1/T)$ for strongly convex regularization functions, where $T$ is the number of iterations. In both cases, the impact of asynchrony on the convergence rate of our algorithm is asymptotically negligible, and a near-linear speedup in the number of workers can be expected. Theoretical results are confirmed in real implementations on a distributed computing infrastructure.
\end{abstract}

%
%

\section{Introduction}
\label{sec:Introduction}

Many optimization problems that arise in machine learning, signal processing, and statistical estimation can be formulated as \textit{regularized stochastic optimization} (also referred to as \textit{stochastic composite optimization}) problems in which one jointly minimizes the expectation of a stochastic loss function plus a possibly nonsmooth regularization term. Examples include Tikhonov and elastic net regularization, Lasso, sparse logistic regression, and support vector machines \cite{Bennett:92,Tibshirani:96,Zou:05,Hastie:09,Shalev:11}.

Stochastic approximation methods such as stochastic gradient descent were among the first algorithms developed for solving stochastic optimization problems~\cite{Robbins:51}.  Recently, these methods have received significant attention due to their simplicity and effectiveness (see, \textit{e.g.},~\cite{Nemirovski:09,Lan2012,Xiao:09,Hu:09,Ghadimi:12,Nedic:14,Needell:14}). In particular, Nemirovski \textit{et. al.}~\cite{Nemirovski:09} demonstrated that for nonsmooth stochastic convex optimization problems, a modified stochastic approximation method, the \textit{mirror descent}, exhibits an unimprovable convergence rate $\mathcal{O}(1/\sqrt{T})$, where $T$ is the number of iterations. Later, Lan~\cite{Lan2012} developed a mirror descent algorithm for stochastic composite convex problems which explicitly accounts for the smoothness of the loss function and achieves the optimal rate. A similar result for the dual averaging method was obtained by Xiao~\cite{Xiao:09}.

The methods for solving stochastic optimization problems cited above are inherently \textit{serial} in the sense that the gradient computations take place on a single processor which has access to the whole dataset. However, it happens more and more often that one single computer is unable to store and handle the amounts of data that we encounter in practical problems. This has caused a strong interest in developing \textit{parallel} optimization algorithms which are able to split the data and distribute the computation across multiple processors or multiple computer clusters (see, \emph{e.g.},~\cite{Tsitsiklis:86,Nedic:01,Zinkevich:09,Zinkevich:10,Lobel:11,Tsianos:12,Recht:13,Li:13,Bianchi:13,McMahan:14,Hong:14,Jaggi:14,Marecek:14,Zhang:14,Zhang:15,Hsieh:15,Wright:15,Richtarik:12} and references therein).

One simple and popular stochastic approximation method is \textit{mini-batching}, where iterates are updated based on the average gradient with respect to multiple data points rather than based on gradients evaluated at a single data at a time. Recently, Dekel \textit{et. al.}~\cite{Dekel:12} proposed a parallel mini-batch algorithm for regularized stochastic optimization problems, in which multiple processors compute gradients in parallel using their own local data, and then aggregate the gradients up a spanning tree to obtain the averaged gradient. While this algorithm can achieve linear speedup in the number of processors, it has the drawback that the processors need to synchronize at each round and, hence, if one of them fails or is slower than the rest, then the entire algorithm runs at the pace of the slowest processor.

In this paper, we propose an \textit{asynchronous} mini-batch algorithm for regularized stochastic optimization problems with smooth loss functions that eliminates the overhead associated with global synchronization. Our algorithm allows multiple processors to work at \textit{different rates}, perform computations \textit{independently} of each other, and update global decision variables using \textit{out-of-date} gradients. A similar model of parallel asynchronous computation was applied to coordinate descent methods for deterministic optimization in~\cite{RRSF:11,liu:14,Liu:15} and mirror descent and dual averaging methods for stochastic optimization in~\cite{AgD:12}. In particular, Agarwal and Duchi~\cite{AgD:12} have analyzed the convergence of asynchronous mini-batch algorithms for smooth stochastic convex problems, and interestingly shown that bounded delays do not degrade the asymptotic convergence. However, they only considered the case where the regularization term is the indicator function of a compact convex set.

We extend the results of~\cite{AgD:12} to general regularization functions (like the $l_1$ norm, often used to promote sparsity), and establish a sharper expected-value type of convergence rate than the one given in~\cite{AgD:12}. Specifically, we make the following contributions:
\begin{itemize}
\item[\textbf{(i)}] For general convex regularization functions, we show that when the constraint set is closed and convex (but not necessarily bounded), the running average of the iterates generated by our algorithm with constant step-sizes converges at rate $\mathcal{O}(1/T)$ to a ball around the optimum. We derive an explicit expression that quantifies how the convergence rate and the residual error depends on loss function properties and algorithm parameters such as the constant step-size, the batch size, and the maximum delay bound $\tau_{\max}$.
\item[\textbf{(ii)}] For general convex regularization functions and compact constraint sets, we prove that the running average of the iterates produced by our algorithm with a time-varying step-size converges to the true optimum (without residual error) at rate
\begin{align*}
\mathcal{O}\left(\frac{(\tau_{\max}+1)^2}{T}+\frac{1}{\sqrt{T}}\right).
\end{align*}
This result improves upon the previously known rate
\begin{align*}
\mathcal{O}\left(\frac{\tau^2_{\max}\log T}{T}+\frac{\tau_{\max}+1}{T}+ \frac{1}{\sqrt{T}}\right)
\end{align*}
for delayed stochastic mirror descent methods with time-varying step-sizes given in~\cite{AgD:12}. In this case, our algorithm enjoys near-linear speedup as long as the number of processors is $\mathcal{O}(T^{1/4})$.
\item[\textbf{(iii)}] When the regularization function is strongly convex and the constraint set is closed and convex, we establish that the iterates converge at rate
\begin{align*}
\mathcal{O}\left(\frac{(\tau_{\max}+1)^4}{T^2}+\frac{1}{{T}}\right).
\end{align*}
If the number of processors is of the order of $\mathcal{O}(T^{1/4})$, this rate is $\mathcal{O}(1/T)$ asymptotically in $T$, which is the best known rate for strongly convex stochastic optimization problems in a serial setting.
\end{itemize}

The remainder of the paper is organized as follows. In Section~\ref{sec:Preliminaries}, we introduce the notation and review some preliminaries that are essential for the development of the results in this paper. In Section~\ref{sec:Problem Statement}, we formulate the problem and discuss our assumptions. The proposed asynchronous mini-batch algorithm and its main theoretical results are presented in Section~\ref{sec:Algorithm}. Computational experience is reported in Section~\ref{sec:Simulation} while Section~\ref{sec:Conclusions} concludes the paper.

%
%

\section{Notation and Preliminaries}
\label{sec:Preliminaries}

\subsection{Notation}

We let $\mathbb{N}$ and $\mathbb{N}_0$ denote the set of natural numbers and the set of natural numbers including zero, respectively. The inner product of two vectors $x, y\in\mathbb{R}^n$ is denoted by $\langle x,y\rangle$. We assume that $\mathbb{R}^n$ is endowed with a norm $\|\cdot\|$, and use $\|\cdot\|_*$  to represent the corresponding dual norm, defined by
\begin{align*}
\|y\|_{*}=\underset{\|x\|\leq 1}{\sup} \;\langle x,y\rangle.
\end{align*}

\subsection{Preliminaries}

Next, we review the key definitions and results necessary for developing the main results of this paper. We start with the definition of a \textit{Bregman distance function}, also referred to as a \textit{prox-function}.

\begin{definition}
\textit{
A function $\omega: X \rightarrow \mathbb{R}$ is called a distance generating function with modulus $\mu_\omega>0$ with respect to norm $\|\cdot \|$, if $\omega$ is continuously differentiable and  $\mu_\omega$-strongly convex with respect to $\|\cdot \|$ over the set $X\subseteq\mathbb{R}^n$. That is, for all $x,y\in X$,
\begin{align*}
\omega (y)\geq \omega (x)+\langle \nabla \omega(x), y-x\rangle+\frac{\mu_\omega}{2}\|y-x \|^2.
\end{align*}
Every distance generating function introduces a corresponding Bregman distance function
\begin{align*}
D_{\omega}(x,y):= \omega(y)-\omega (x)-\langle \nabla \omega(x), y-x\rangle.
\end{align*}}
\end{definition}

For example, choosing $\omega(x)=\frac{1}{2}\| x\|_2^2$, which is $1$-strongly convex with respect to the $l_2$-norm over any convex set $X$, would result in $D_{\omega}(x,y)=\frac{1}{2}\| x-y\|_2^2$. Another common example of distance generating functions is the entropy function
\begin{align*}
\omega(x)=\sum_{i=1}^n x_i \log x_i,
\end{align*}
which is $1$-strongly convex with respect to the $l_1$-norm over the standard simplex
\begin{align*}
\Delta:=\left\{x\in \mathbb{R}^n\;\biggl |\; \sum_{i=1}^n x_i = 1,\; x\geq 0\right\},
\end{align*}
and its associated Bregman distance function is
\begin{align*}
D_{\omega}(x,y)=\sum_{i=1}^n y_i \log \frac{y_i}{x_i}.
\end{align*}
The main motivation to use a generalized distance generating function, instead of the usual Euclidean distance function, is to design optimization algorithms that can take advantage of the geometry of the feasible set (see, \emph{e.g.},~\cite{Beck:03,Nemirovski:09,Tseng:10,Duchi:10}).

\begin{remark}
\label{Remark 1}
The strong convexity of the distance generating function $\omega$ always ensures that
\begin{align*}
 D_{\omega}(x,y)\geq \frac{\mu_\omega}{2}\|y-x \|^2,\quad \forall x,y\in X,
\end{align*}
and $D_{\omega}(x,y)=0$ if and only if $x=y$.
\end{remark}

\begin{remark}
Throughout the paper, there is no loss of generality to assume that $\mu_\omega=1$. Indeed, if $\mu_\omega \neq 1$, we can choose the scaled function $\overline{\omega}(x)=\frac{1}{\mu_\omega} \omega(x)$, which has modulus $\overline{\mu}_\omega=1$,  to generate the Bregman distance function.
\end{remark}

The following definition introduces \textit{subgradients} of proper convex functions.

\begin{definition}
\textit{
For a convex function $\Psi:\mathbb{R}^{n}\rightarrow \mathbb{R}\cup \{+\infty\}$, a vector $s\in\mathbb{R}^n$ is called a subgradient of $\Psi$ at $x \in\mathbb{R}^n$ if
\begin{align*}
\Psi(y)\geq \Psi(x)+\langle s,y-x\rangle,\quad \forall y\in\mathbb{R}^n.
\end{align*}
The set of all subgradients of $\Psi$ at $x$ is called the subdifferential of $\Psi$ at $x$, and is denoted by $\partial \Psi(x)$.}
\end{definition}

%
%

\section{Problem Setup}
\label{sec:Problem Statement}

We consider stochastic convex optimization problems of the form
\begin{align}
\underset{x}{\textup{minimize}}\; \phi(x):={\mathbb{E}}_\xi\bigl[F(x,\xi)\bigr]+\Psi(x).
\label{Optimization Problem 1}
\end{align}
Here, $x\in\mathbb{R}^n$ is the decision variable, $\xi$ is a random vector whose probability distribution $\mathcal{P}$ is supported on a set $\Xi \subseteq~\mathbb{R}^m$, $F(\cdot,\xi)$ is convex and differentiable for each $\xi\in\Xi$, and $\Psi(x)$ is a proper convex function that may be nonsmooth and extended real-valued. Let us define
\begin{align}
f(x):={\mathbb{E}}_\xi\bigl[F(x,\xi)\bigr]=\int_{\Xi}F(x,\xi)d \mathcal{P}(\xi).
\label{Expectation function}
\end{align}
Note that the expectation function $f$ is convex, differentiable, and $\nabla f(x)=\mathbb{E}_\xi [\nabla_x F(x,\xi)]$~\cite{Rockafellar:82}. We use $X^{\star}$ to denote the set of optimal solutions of Problem~\eqref{Optimization Problem 1} and $\phi^{\star}$ to denote the corresponding optimal value.

A difficulty when solving optimization problem~\eqref{Optimization Problem 1} is that the distribution ${\mathcal P}$ is often unknown, so the expectation~\eqref{Expectation function} cannot be computed. This situation occurs frequently in data-driven applications such as machine learning. To support these applications, we do not assume knowledge of $f$ (or of ${\mathcal P}$), only access to a stochastic oracle. Each time the oracle is queried with an $x\in \mathbb{R}^n$, it generates an independent and identically distributed (i.i.d.) sample $\xi$ from ${\mathcal P}$ and returns $\nabla_x F(x,\xi)$.

We also impose the following assumptions on Problem~\eqref{Optimization Problem 1}.
\begin{assumption}[\textbf{\small{Existence of a minimum}}\normalsize]
\label{Assumption 1}
The optimal set $X^{\star}$ is nonempty.
\end{assumption}

\begin{assumption}[\textbf{\small Lipschitz continuity of} $F$\normalsize]
\label{Assumption 2}
For each $\xi\in \Xi$, the function $F(\cdot,\xi)$ has Lipschitz continuous gradient with constant $L$. That is, for all $y,z\in \mathbb{R}^n$,
\begin{align*}
\| \nabla_x F(y,\xi)- \nabla_x F(z,\xi) \|_{*}\leq  L \|y-z\|.
\end{align*}
\end{assumption}

Note that under Assumption~\ref{Assumption 2}, $\nabla f(x)$ is also Lipschitz continuous with the same constant $L$~\cite{Xiao:09}.

\begin{assumption}[\textbf{\small{Bounded gradient variance}}\normalsize]
\label{Assumption 3}
There exists a constant $\sigma\geq 0$ such that
\begin{align*}
\mathbb{E}_\xi\bigl[\|\nabla_x F(x,\xi)-\nabla f(x)\|_*^2\bigr]\leq \sigma^2,\quad \forall x\in \mathbb{R}^n.
\end{align*}
\end{assumption}

\begin{assumption}[\textbf{\small Closed effective domain of} $\Psi$\normalsize]
\label{Assumption 4}
The function $\Psi$ is simple and lower semi-continuous, and its effective domain, $\textup{dom}\;\Psi=\{x\in\mathbb{R}^n\;|\;\Psi(x)<+\infty\}$, is closed.
\end{assumption}

Possible choices of $\Psi$ include:

\begin{itemize}
\item \textit{Unconstrained smooth minimization}: $\Psi(x)=0$.
\item \textit{Constrained smooth minimization}: $\Psi$ is the indicator function of a non-empty closed convex set $C\subseteq \mathbb{R}^n$, \textit{i.e.},
\begin{align*}
\Psi(x)=I_C(x):=& \left\{
\begin{array}[l]{ll}
\;\;0,\quad &\textup{if}\;\;x\in C,\\
+\infty,\quad &\textup{otherwise}.
\end{array}
\right.
\end{align*}
\item $l_1$-\textit{regularized minimization}: $\Psi(x)=\lambda \|x\|_1$ with $\lambda>0$.
\item \textit{Constrained} $l_1$-\textit{regularized minimization}: In this case, $\Psi(x)=\lambda \|x\|_1+I_C(x)$ with $\lambda>0$.
\end{itemize}

Several practical problems in machine learning, statistical applications, and signal processing satisfy Assumptions~\ref{Assumption 1}--\ref{Assumption 4} (see, \textit{e.g.},~\cite{Tibshirani:96,Zou:05,Hastie:09}). One such example is $l_1$-\textit{regularized logistic regression} for sparse binary classification. We are then given a large number of observations
\begin{align*}
\bigl\{\xi_j=(a_j,b_j)\;|\;a_j\in\mathbb{R}^n,\;b_j\in\{-1,+1\},\;j=1,\ldots,m\bigr\},
\end{align*}
drawn i.i.d. from an unknown distribution $\mathcal{P}$, and want to solve the minimization problem~\eqref{Optimization Problem 1} with
\begin{align*}
F(x,\xi)=\log\bigl(1+\textup{exp}(-b\langle a,x\rangle\bigr),
\end{align*}
and $\Psi(x)=\lambda \|x\|_1$. The role of $l_1$ regularization is to produce sparse solutions.

One approach for solving Problem~\eqref{Optimization Problem 1} is the \textit{serial mini-batch method} based on the mirror descent scheme~\cite{Dekel:12}. Given a point $x\in \textup{dom}\; \Psi$, a single processor updates the decision variable $x$ by sampling $b$ i.i.d. random variables $\xi_1,\ldots,\xi_b$ from $\mathcal{P}$, computing the averaged stochastic gradient
\begin{align*}
{g}_{\textup{ave}}=\frac{1}{b}\sum_{i=1}^b \nabla_x F(x,\xi_i),
\end{align*}
and performing the composite mirror descent update
\begin{align*}
x\leftarrow\underset{z}{\textup{argmin}}\biggl\{\bigl\langle {g}_{\textup{ave}},z\bigr\rangle+\Psi(z)+\frac{1}{\gamma}D_\omega(x,z)\biggr\},
\end{align*}
where $\gamma$ is a positive step-size parameter. Under Assumptions~\ref{Assumption 1}--\ref{Assumption 4} and choosing an appropriate step-size, this algorithm is guaranteed to converge to the optimum~\cite[Theorem 9]{Dekel:12}. However, in many emerging applications, such as large-scale machine learning and statistics, the size of dataset is so huge that it cannot fit on one machine. Hence, we need optimization algorithms that can be conveniently and efficiently executed in parallel on multiple processors.

%
%

\section{An Asynchronous Mini-Batch Algorithm}
\label{sec:Algorithm}

In this section, we will present an \textit{asynchronous} mini-batch algorithm that exploits multiple processors to solve Problem~\eqref{Optimization Problem 1}. We characterize the iteration complexity and the convergence rate of the proposed algorithm, and show that these compare favourably with the state of the art.

\subsection{Description of Algorithm}
\label{sec:Description of Algorithm}

We assume \textit{p} processors have access to a shared memory for the decision variable $x$. The processors may have different capabilities (in terms of processing power and access to data) and are able to update $x$ without the need for coordination or synchronization. Conceptually, the algorithm lets each processor run its own stochastic composite mirror descent process, repeating the following steps:

\begin{enumerate}
\item Read $x$ from the shared memory and load it into the local storage location $\widehat{x}$;
\item Sample $b$ i.i.d random variables $\xi_1,\ldots,\xi_b$ from the distribution $\mathcal{P}$;
\item Compute the averaged stochastic gradient vector
\begin{align*}
{\widehat{g}}_{\textup{ave}}=\frac{1}{b}\sum_{i=1}^b \nabla_x F(\widehat{x},\xi_i);
\end{align*}
\item Update current $x$ in the shared memory via
\begin{align*}
x\leftarrow\underset{z}{\textup{argmin}}\biggl\{\bigl\langle {\widehat{g}}_{\textup{ave}},z\bigr\rangle+\Psi(z)+\frac{1}{\gamma}D_\omega(x,z)\biggr\}.
\end{align*}
\end{enumerate}

The algorithm can be implemented in many ways as depicted in Figure~\ref{fig:implementation}. One way is to  consider the $p$ processors as peers that each execute the four-step algorithm independently of each other and only share the global memory for storing $x$. In this case, each processor reads the decision vector twice in each round: once in the first step (before evaluating the averaged gradient), and once in the last step (before carrying out the minimization). To ensure correctness, Step 4 must be an atomic operation, where the executing processor puts a write lock on the global memory until it has written back the result of the minimization (cf. Figure~\ref{fig:implementation}, left). The algorithm can also be executed in a master-worker setting. In this case, each of the worker nodes retrieves $x$ from the master in Step~1 and returns the averaged gradient to the master in Step 3; the fourth step (carrying out the minimization) is executed by the master (cf. Figure~\ref{fig:implementation}, right)

\begin{figure}[!ht]
    \centering
    \includegraphics[width=0.8\columnwidth]{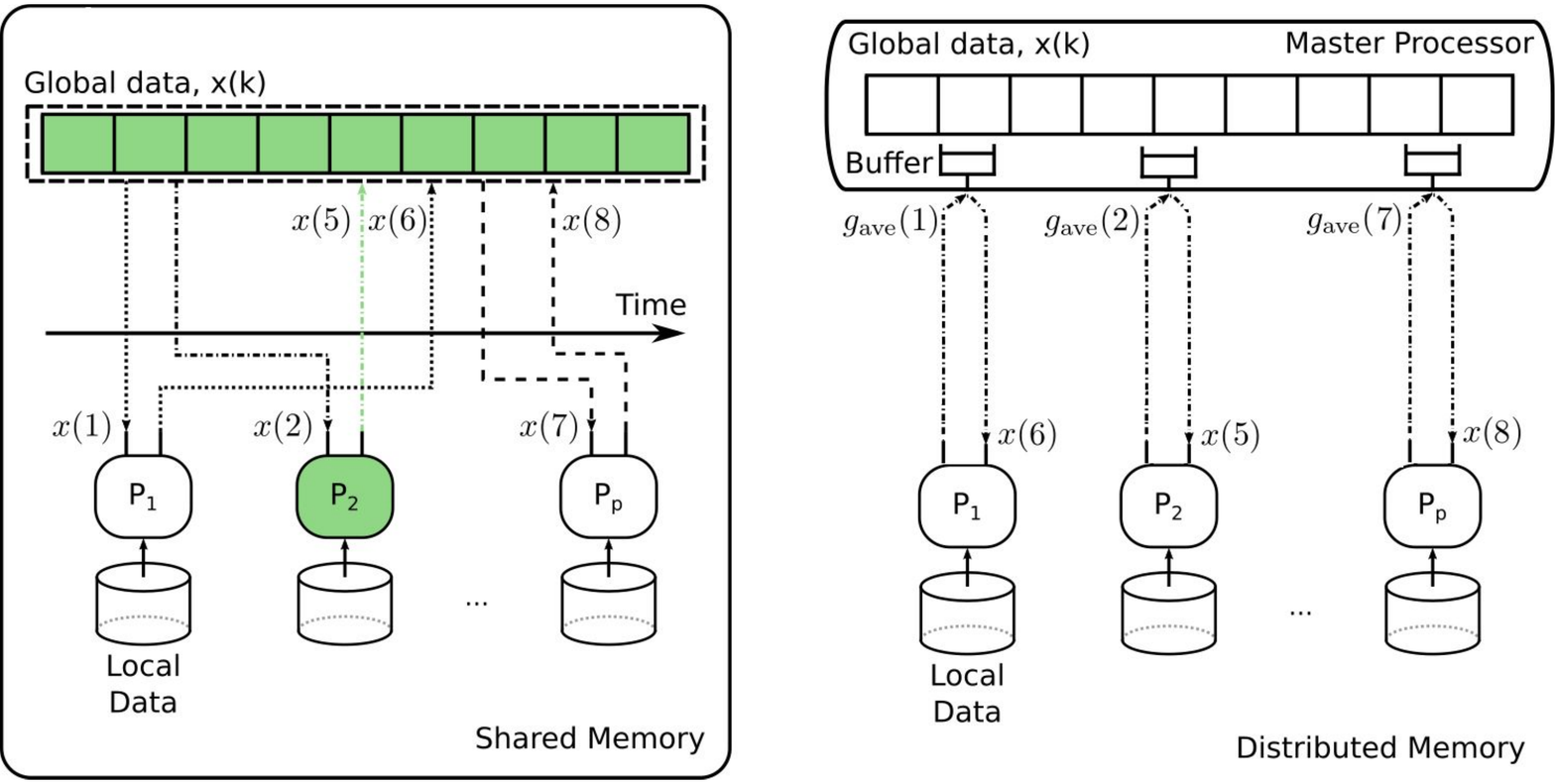}
     \caption{Illustration of two conceptually different realizations of Algorithm~\ref{Algorithm 1}: $(1)$ a shared memory implementation (left); $(2)$ a master-worker implementation (right). In the shared memory setting shown to the left, processor $P_{2}$ reads $x(2)$ from the shared memory and computes the averaged gradient vector ${g}_{\textup{ave}}(2)=\frac{1}{b}\sum_{i=1}^b \nabla_x F\bigl({x}(2),\xi_i\bigr)$. As the processors are being run without synchronization, $x(3)$ and $x(4)$ are written to the shared memory by other processors while $P_{2}$ is evaluating ${g}_{\textup{ave}}(2)$. The figure shows a snapshot of the algorithm at time instance $k = 5$, at which the shared memory is locked by $P_{2}$ to read the current $x$, \textit{i.e.} $x(4)$, to update it using the out-of-date gradient ${g}_{\textup{ave}}(2)$, and write $x(5)$ to the memory. In the master-worker setting illustrated to the right, workers evaluate averaged gradient vectors in parallel and send their computations to buffers on the master processor, which is the sole entity with access to the global memory.  The master performs an update using (possibly) out-of-date gradients and passes the updated decision vector $x$ back to the workers.}
    \label{fig:implementation}
\end{figure}

Independently of how we choose to implement the algorithm,  processors may work at different rates: while one processor updates the decision vector (in the shared memory setting) or send its averaged gradient to the master (in the master-worker setting), the others are generally busy computing averaged gradient vectors. The processors that perform gradient evaluations do not need to be aware of updates to the decision vector, but can continue to operate on stale information about $x$. Therefore, unlike \textit{synchronous} parallel mini-batch algorithms~\cite{Dekel:12}, there is no need for processors to wait for each other to finish the gradient computations. Moreover, the value $\widehat{x}$ at which the average of gradients is evaluated by a processor may differ from the value of $x$ to which the update is applied.

Algorithm~\ref{Algorithm 1} describes the $p$ asynchronous processes that run in parallel. To describe the progress of the overall optimization process, we introduce a counter $k$ that is incremented each time $x$ is updated. We let $d(k)$ denote the time at which $\widehat{x}$ used to compute the averaged gradient involved in the update of $x(k)$ was read from the shared memory.  It is clear that $0\leq d(k)\leq k$ for all $k\in \mathbb{N}_0$. The value
\begin{align*}
\tau(k):=k-d(k)
\end{align*}
can be viewed as the delay between reading and updating for processors and captures the staleness of the information used to compute the average of gradients for the \textit{k}-th update.  We assume that the delay is not too long, \textit{i.e.}, there is a nonnegative integer $\tau_{\max}$ such that
\begin{align*}
0\leq \tau(k)\leq \tau_{\max}.
\end{align*}
The value of $\tau_{\max}$ is an indicator of the asynchronism in the algorithm and in the execution platform. In practice, $\tau_{\max}$ will depend on the number of parallel processors used in the algorithm~\cite{RRSF:11,liu:14,Liu:15}. Note that the cyclic-delay mini-batch algorithm~\cite{AgD:12}, in which the processors are ordered and each updates the decision variable under a fixed schedule, is a special case of Algorithm~\ref{Algorithm 1} where $d(k)=k-p+1$, or, equivalently, $\tau(k)=p-1$ for all $k$.

{
\begin{algorithm} [t]
\caption{{Asynchronous Mini-batch Algorithm (running on each processor)}}
\label{Algorithm 1}
\begin{algorithmic}[1]
\STATE  \textbf{Inputs}: positive step-sizes $\{\gamma(k)\}_{k\in \mathbb{N}_0}$; batch size $b\in\mathbb{N}$.
\STATE \textbf{Initialization:} $x(0)\in \textup{dom}\;\Psi$; $k=0$.
\REPEAT
\STATE receive inputs $\xi_1,\ldots,\xi_b$ sampled i.i.d. from distribution $\mathcal{P}$;
\begin{align}
&{g}_{\textup{ave}}\bigl(d(k)\bigr)\leftarrow\frac{1}{b}\sum_{i=1}^b \nabla_x F\bigl({x}(d(k)),\xi_i\bigr) ;\nonumber\\
&x(k+1)\leftarrow\underset{z}{\textup{argmin}}\biggl\{\bigl\langle {g}_{\textup{ave}}\bigl(d(k)\bigr),z\bigr\rangle+\Psi(z)+\frac{1}{\gamma(k)}D_\omega(x(k),z)\biggr\}\label{Async. Iteration}\\
&k\leftarrow k+1\nonumber;
\end{align}
\UNTIL {termination test satisfied}
\end{algorithmic}
\end{algorithm}
}

\subsection{Convergence Rate for General Convex Regularization}
\label{sec:Convex Regularization}

The following theorem establishes convergence properties of Algorithm~\ref{Algorithm 1} when a constant step-size is used.

\begin{theorem}
\label{Theorem 1}
\textit{
Let Assumptions~\ref{Assumption 1}--\ref{Assumption 4} hold. Assume also that for all $k\in\mathbb{N}_0$,
\begin{align}
\gamma(k)=\gamma\in \left(0,\;\frac{1}{L(\tau_{\max}+1)^2}\right).
\label{Fixed Step-size}
\end{align}
Then, for every $T\in \mathbb{N}$ and any optimizer $x^{\star}$ of~\eqref{Optimization Problem 1},  we have}
\begin{align*}
\mathbb{E}\bigl[\phi\bigl({x}_{\textup{ave}}(T)\bigr)\bigr]-\phi^{\star}\leq \frac{D_\omega\bigl(x(0),x^{\star}\bigr)}{\gamma T}+\frac{\gamma c\sigma^2}{2b\bigl(1-\gamma L(\tau_{\max}+1)^2\bigr)}\;,
\end{align*}
\textit{where ${x}_{\textup{ave}}(T)$ is the Ces{\'a}ro average of the iterates, \textit{i.e.},
\begin{align*}
{x}_{\textup{ave}}(T):=\frac{1}{T}\sum_{k=1}^T x(k).
\end{align*}
Furthermore, $b$ is the batch size, the expectation is taken with respect to all random variables $\{\xi_i(k)\;|\;i=1,\ldots,b,\;k=0,\ldots,T-1\}$, and $c\in[1,b]$ is given by}
\begin{align*}
 c = \left\{
  \begin{array}{ll}
    1, & \textup{if}\;\|\cdot\|_*=\|\cdot\|_2,\\
    {2\max_{\|x\|\leq 1}\omega(x)},&\textup{otherwise}.
  \end{array}
   \right.
\end{align*}
\end{theorem}

\begin{proof}
See Appendix~\ref{Theorem:1}.
\end{proof}

Theorem~\ref{Theorem 1} demonstrates that for any constant step-size $\gamma$ satisfying~\eqref{Fixed Step-size}, the running average of iterates generated by Algorithm~\ref{Algorithm 1} will converge in expectation to a ball around the optimum at a rate of $\mathcal{O}(1/T)$. The convergence rate and the residual error depend on the choice of $\gamma$: decreasing $\gamma$ reduces the residual error, but it also results in a slower convergence. We now describe a possible strategy for selecting the constant step-size. Let $T_{\epsilon}$ be the total number of iterations necessary to achieve $\epsilon$-optimal solution to Problem~\eqref{Optimization Problem 1}, that is, $\mathbb{E}\bigl[\phi\bigl({x}_{\textup{ave}}(T)\bigr)\bigr]-\phi^{\star}\leq \epsilon$ when $T\geq T_\epsilon$. If we pick
\begin{align}
\gamma=\frac{\epsilon}{L\epsilon(\tau_{\max}+1)^2+c\sigma^2/b},
\label{Step-size Constant}
\end{align}
it follows from Theorem~\ref{Theorem 1} that the corresponding ${x}_{\textup{ave}}(T)$ satisfies
\begin{align*}
\mathbb{E}\bigl[\phi\bigl({x}_{\textup{ave}}(T)\bigr)\bigr]-\phi^{\star}\leq \frac{\epsilon_0}{T}\left( L(\tau_{\max}+1)^2+\frac{c\sigma^2}{b\epsilon}\right)+\frac{\epsilon}{2},
\end{align*}
\normalsize
where $\epsilon_0=D_\omega\bigl(x(0),x^{\star}\bigr)$. This inequality tells us that if the first term on the right-hand side is less than $\epsilon/2$, \textit{i.e.}, if
\begin{align*}
T\geq T_{\epsilon}:= 2\epsilon_0 \left(\frac{ L(\tau_{\max}+1)^2}{\epsilon}+\frac{c\sigma^2}{b\epsilon^2}\right),
\end{align*}
then $\mathbb{E}\bigl[\phi\bigl({x}_{\textup{ave}}(T)\bigr)\bigr]-\phi^{\star}\leq \epsilon$. Hence, the iteration complexity of Algorithm~\ref{Algorithm 1} with the step-size choice~\eqref{Step-size Constant} is given by
\begin{align}
\mathcal{O}\left(\frac{L(\tau_{\max}+1)^2}{\epsilon}+\frac{c\sigma^2}{b\epsilon^2}\right).
\label{Complexity Parallel}
\end{align}

As long as the maximum delay bound $\tau_{\max}$ is of the order $1/\sqrt{\epsilon}$, the first term in~\eqref{Complexity Parallel} is asymptotically negligible, and hence the iteration complexity of Algorithm~\ref{Algorithm 1} is asymptotically $\mathcal{O}(c\sigma^2/b{\epsilon}^2)$, which is exactly the iteration complexity achieved by the mini-batch algorithm for solving stochastic convex optimization problems in a serial setting~\cite{Dekel:12}. As discussed before, $\tau_{\max}$ is related to the number of processors used in the algorithm. Therefore, if the number of processors is of the order of $\mathcal{O}(1/\sqrt{\epsilon})$, parallelization does not appreciably degrade asymptotic convergence of Algorithm~\ref{Algorithm 1}. Furthermore, as $p$ processors are being run in parallel, updates occur roughly $p$ times as quickly and in time scaling as $T/p$, the processors may compute $T$ averaged gradient vectors (instead of $T/p$ vectors). This means that the near-linear speedup in the number of processors can be expected.

\begin{remark}
Another strategy for the selection of the constant step-size in Algorithm~\ref{Algorithm 1} is to use $\gamma$ that depends on the prior knowledge of the number of iterations to be performed. More precisely, assume that the number of iterations is fixed in advance, say equal to $T_F$. By choosing  $\gamma$ as
\begin{align*}
\gamma = \frac{1}{L(\tau_{\max}+1)^2+\alpha\sqrt{T_F}}\;,
\end{align*}
for some $\alpha>0$, it follows from Theorem~\ref{Theorem 1} that the running average of the iterates after $T_F$ iterations satisfies
\begin{align*}
\mathbb{E}\bigl[\phi\bigl({x}_{\textup{ave}}(T_F)\bigr)\bigr]-\phi^{\star}\leq& \frac{L(\tau_{\max}+1)^2 D_\omega\bigl(x(0),x^{\star}\bigr)}{T_F}+\frac{1}{\sqrt{T_F}}\left(\alpha D_\omega\bigl(x(0),x^{\star}\bigr)+\frac{c\sigma^2}{2\alpha b}\right).
\end{align*}
It is easy to verify that the optimal choice of $\alpha$, which minimizes the second term on the right-hand-side of the above inequality, is
\begin{align*}
\alpha^{\star}=\frac{\sigma\sqrt{c}}{\sqrt{2bD_\omega\bigl(x(0),x^{\star}\bigr)}}.
\end{align*}
With this choice of $\alpha$, we then have
\begin{align*}
\mathbb{E}\bigl[\phi\bigl({x}_{\textup{ave}}(T_F)\bigr)\bigr]-\phi^{\star}\leq& \frac{L(\tau_{\max}+1)^2 D_\omega\bigl(x(0),x^{\star}\bigr)}{T_F}+\frac{\sigma\sqrt{2cD_\omega\bigl(x(0),x^{\star}\bigr)}}{\sqrt{bT_F}}.
\end{align*}
In the case that $\tau_{\max}=0$, the preceding guaranteed bound reduces to the one obtained in~\cite[Theorem 1]{Lan2012} for the serial stochastic mirror descent algorithm with constant step-sizes. Note that in order to implement Algorithm~\ref{Algorithm 1} with the optimal constant step-size policy, we need to estimate an upper bound on $D_\omega\bigl(x(0),x^{\star}\bigr)$, since $D_\omega\bigl(x(0),x^{\star}\bigr)$ is usually unknown.
\end{remark}

The following theorem characterizes the convergence of Algorithm~\ref{Algorithm 1} with a time-varying step-size sequence when $\textup{dom}\; \Psi$ is \textit{bounded} in addition to being closed and convex.

\begin{theorem}
\label{Theorem 2}
\textit{
Suppose that Assumptions~\ref{Assumption 1}--\ref{Assumption 4} hold. In addition, suppose that $\textup{dom}\; \Psi$ is compact and that $D_{\omega}(\cdot,\cdot)$ is bounded on $\textup{dom}\; \Psi$. Let
\begin{align*}
R^2=\max_{x,y\in\textup{dom}\; \Psi}\; D_{\omega}(x,y).
\end{align*}
If $\{\gamma(k)\}_{k\in\mathbb{N}_0}$ is set to $\gamma(k)^{-1}=L(\tau_{\max}+1)^2+\alpha(k)$ with
\begin{align*}
\alpha(k)=\frac{\sigma\sqrt{c}\sqrt{k+1}}{R\sqrt{b}},
\end{align*}
then the Ces{\'a}ro average of the iterates generated by Algorithm~\ref{Algorithm 1} satisfies
\begin{align*}
\mathbb{E}\bigl[\phi\bigl({x}_{\textup{ave}}(T)\bigr)\bigr]-\phi^{\star}\leq \frac{LR^2(\tau_{\max}+1)^2}{T}+\frac{2\sigma R \sqrt{c}}{\sqrt{bT}},
\end{align*}
for all $T\in\mathbb{N}$.}
\end{theorem}

\begin{proof}
See Appendix~\ref{Theorem:2}.
\end{proof}

The time-varying step-size $\gamma(k)$, which ensures the convergence of the algorithm, consists of two terms: the time-varying term $\eta(k)$ should control the errors from stochastic gradient information while the role of the constant term ($L(\tau_{\max}+1)^2$) is to decrease the effects of asynchrony (bounded delays) on the convergence of the algorithm. According to Theorem~\ref{Theorem 2}, in the case that $\tau_{\max}=\mathcal{O}(T^{1/4})$, the delay becomes increasingly harmless as the algorithm progresses and the expected function value evaluated at ${x}_{\textup{ave}}(T)$ converges asymptotically at a rate $\mathcal{O}(1/\sqrt{T})$, which is known to be the best achievable rate of the mirror descent method for nonsmooth stochastic convex optimization problems~\cite{Nemirovski:09}.

For the special case of the optimization problem~\eqref{Optimization Problem 1} where $\Psi$ is restricted to be the indicator function of a compact convex set, Agarwal and Duchi~\cite[Theorem 2]{AgD:12} showed that the convergence rate of the delayed stochastic mirror descent method with time-varying step-size is
\begin{align*}
\mathcal{O}\left(\frac{L R^2+RG\tau_{\max}}{T}+\frac{\sigma R \sqrt{c}}{\sqrt{Tb}}+\frac{LR^2 G^2\tau^2_{\max} b \log T}{c\sigma^2 T}\right),
\end{align*}
where $G$ is the maximum bound on $\sqrt{\mathbb{E}[\|\nabla_x F(x,\xi)\|_*^2]}$. Comparing with this result, instead of a asymptotic penalty of the form $\mathcal{O}(\tau^2_{\max} \log T/T)$ due to the delays, we have the  penalty $\mathcal{O}(\tau^2_{\max}/T)$, which is much smaller for large $T$. Therefore, not only do we extend the result of~\cite{AgD:12} to  general regularization functions, but we also obtain a sharper guaranteed convergence rate than the one presented in~\cite{AgD:12}.

\subsection{Convergence Rate for Strongly Convex Regularization}
\label{sec:Strongly Convex Regularization}

In this subsection, we restrict our attention to stochastic composite optimization problems with strongly convex regularization terms. Specifically, we assume that $\Psi$ is $\mu_{\Psi}$-strongly convex with respect to $\|\cdot \|$, that is, for any $x,y\in\textup{dom}\; \Psi$,
\begin{align*}
\Psi(y)\geq \Psi(x)+\langle s,y-x\rangle+\frac{\mu_\Psi}{2}\|y-x \|^2,\quad \forall s\in\partial \Psi(x).
\end{align*}
Examples of the strongly convex function $\Psi$ include:
\begin{itemize}
\item $l_2$-\textit{regularization}: $\Psi(x)=(\rho/2) \|x\|_2^2$ with $\rho>0$.\vspace{0.2mm}
\item \textit{Elastic net regularization}: $\Psi(x)=\lambda \|x\|_1+(\rho/2) \|x\|_2^2$ with $\lambda>0$ and $\rho>0$.
\end{itemize}

\begin{remark}
The strong convexity of $\Psi$ implies that Problem~\eqref{Optimization Problem 1} has a unique minimizer $x^{\star}$~\cite[Corollary 11.16]{Bauschke:11}.
\end{remark}

In order to derive the convergence rate of Algorithm~\ref{Algorithm 1} for solving~\eqref{Optimization Problem 1} with a strongly convex regularization term, we need to assume that the Bregman distance function $D(x,y)$ used in the algorithm satisfies the next assumption.

\begin{assumption}[\textbf{\small Quadratic growth condition}\normalsize]
\label{Assumption 5}
For all $x,y\in\textup{dom}\; \Psi$, we have
\begin{align*}
D_{\omega}(x,y)\leq \frac{Q}{2}\| x-y\|^2,
\end{align*}
with $Q\geq \mu_\omega$.
\end{assumption}
For example, if $\omega(x)=\frac{1}{2}\| x\|_2^2$, then  $D_{\omega}(x,y)=\frac{1}{2}\| x-y\|_2^2$ and $Q=1$. Note that Assumption~\ref{Assumption 5} will automatically hold when the distance generating function $\omega$ has Lipschitz continuous gradient with a constant $Q$~\cite{Nedic:14}.

The associated convergence result now reads as follows.

\begin{theorem}
\label{Theorem 3}
\textit{
Suppose that the regularization function $\Psi$ is $\mu_{\Psi}$-strongly convex and that Assumptions~\ref{Assumption 2}--\ref{Assumption 5} hold. If $\{\gamma(k)\}_{k\in\mathbb{N}_0}$ is set to $\gamma(k)^{-1}=2L(\tau_{\max}+1)^2+\beta(k)$ with
\begin{align*}
\beta(k)=\frac{\mu_\Psi}{3Q}\bigl(k+\tau_{\max}+1\bigr),
\end{align*}
then the iterates produced by Algorithm~\ref{Algorithm 1} satisfies}
\begin{align*}
\mathbb{E}\bigl[\| x(T)&- x^\star\|^2\bigr]\leq \frac{2\left(\frac{6LQ}{\mu_\Psi}+1\right)^2(\tau_{\max}+1)^4}{(T+1)^2}D_\omega\bigl(x(0),x^\star\bigr)+\frac{18c\sigma^2 Q^2}{b\mu^2_\Psi (T+1)},
\end{align*}
\textit{for all $T\in\mathbb{N}$}.
\end{theorem}

\begin{proof}
See Appendix~\ref{Theorem:3}.
\end{proof}

An interesting point regarding Theorem~\ref{Theorem 3} is that for solving stochastic composite optimization problems with strongly convex regularization functions, the maximum delay bound $\tau_{\max}$ can be as large as $\mathcal{O}(T^{1/4})$ without affecting the asymptotic convergence rate of Algorithm~\ref{Algorithm 1}. In this case, our asynchronous mini-batch algorithm converges asymptotically at a rate of $\mathcal{O}(1 / T)$, which matches the best known rate achievable in a serial setting.

%
%

\section{Experimental Results}
\label{sec:Simulation}

We have developed a complete master-worker implementation of our algorithm in C/++ using the Massage Passing
Interface libraries (OpenMPI). Although we argued in Section~\ref{sec:Algorithm} that Algorithm~\ref{Algorithm 1} can
be implemented using atomic operations on shared-memory computing architectures, we have chosen the MPI
implementation due to its flexibility in scaling the problem to distributed-memory environments.

We evaluated our algorithm on a document classification problem using the text categorization dataset \texttt{rcv1}~\cite{Lewis:04}. This dataset consists of $m \approx 800000$ documents, with $n\approx 50000$ unique stemmed tokens spanning 103 topics. Out of these topics, we decided to classify sports-related documents. To this end, we trained a sparse (binary) classifier by solving the following  $l_1$-regularized logistic regression problem
\begin{align*}
& \underset{x}{\textup{minimize}} \hspace{1cm}  \mathbb{E}_{\lbrace \left(a_i, b_i\right) \rbrace}
\left[ \log \left( 1 + \exp\left( -b_i \langle a_i, x \rangle \right) \right) \right]
+ \lambda \|x\|_1 \\
& \textup{subject to}  \hspace{1cm}  \|x\|_{2} \leq R\,.
\end{align*}
Here, $a_i \in \mathbb{R}^{n}$ is the sparse vector of token weights assigned to each document, and $b_i \in \lbrace
-1, 1 \rbrace$ indicates whether a selected document is sports-related, or not ($b_i$ is $1$ if the document is about sport, $-1$ otherwise). To evaluate scalability, we used both
the training and test sets available when solving the optimization problem. We  implemented Algorithm~\ref{Algorithm 1}
with time-varying step-sizes, and used a batch size of 1000 documents. The regularization parameter was set to $\lambda = 0.01$, and the algorithm was run until a fixed tolerance $\epsilon$ was met.

Figure~\ref{fig:speedup} presents the achieved relative speedup of the algorithm with respect to the number
of workers used. The relative speedup of the algorithm on $p$ processors is defined as $S(p) = t_{1}/t_{p}$,
where $t_{1}$ and $t_{p}$ are the time it takes to run the corresponding algorithm (to $\epsilon$-accuracy) on 1 and $p$ processing units, respectively. We observe a near-linear relative speedup, consistent with our theoretical results. The timings are averaged over 10 Monte Carlo runs.
\begin{figure}[!ht]
    \centering
    \includegraphics[width=3in]{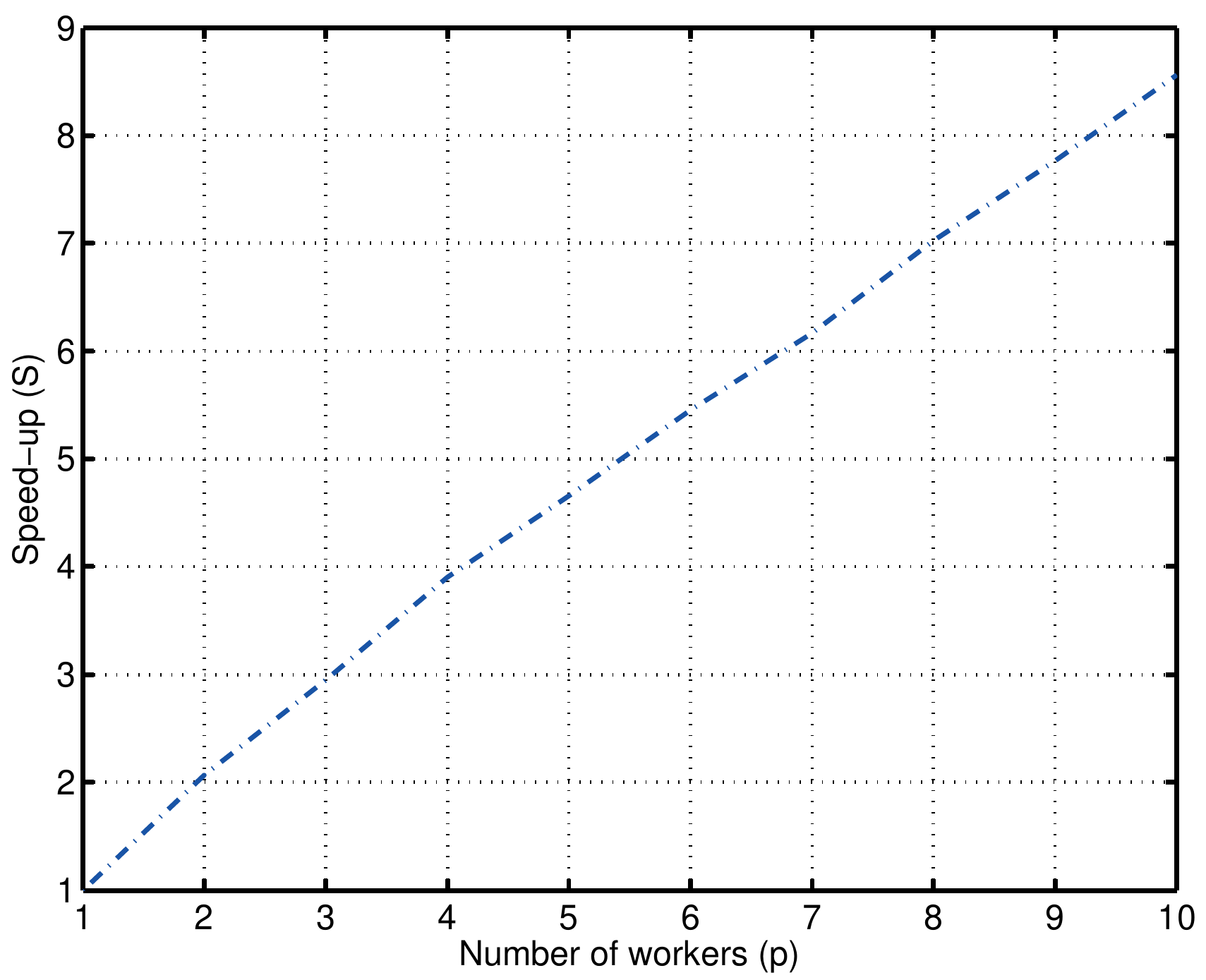}
    \caption{Speedup of Algorithm~\ref{Algorithm 1} with respect to the number of workers.}
    \label{fig:speedup}
\end{figure}

%
%

\section{Conclusions}
\label{sec:Conclusions}

We have proposed an asynchronous mini-batch algorithm that exploits multiple processors to solve regularized stochastic optimization problems with smooth loss functions. We have established that for closed and convex constraint sets, the iteration complexity of the algorithm with constant step-sizes is asymptotically $\mathcal{O}(1\slash\epsilon^2)$. For compact constraint sets, we have proved that the running average of the iterates generated by our algorithm with time-varying step-size converges to the optimum at a rate $\mathcal{O}(1/\sqrt{T})$. When the regularization function is strongly convex and the constraint set is closed and convex, the algorithm achieves the rate of the order $\mathcal{O}(1/T)$. We have shown that the penalty in convergence rate of the algorithm due to asynchrony is asymptotically negligible and a near-linear speedup in the number of processors can be expected. Our computational experience confirmed the theory.

%
%

\appendix
In this section, we prove the main results of the paper, namely, Theorems~\ref{Theorem 1}--\ref{Theorem 3}. We first state three key lemmas which are instrumental in our argument.

%
%

The following result establishes an important recursion for the iterates generated by Algorithm~\ref{Algorithm 1}.

\begin{lemma}
\label{Lemma 1}
Suppose Assumptions~\ref{Assumption 1}--\ref{Assumption 4} hold. Then, the iterates $\{x(k)\}_{k\in\mathbb{N}_0}$ generated by Algorithm~\ref{Algorithm 1} satisfy
\small
\begin{align}
\phi\bigl(x(k+1)\bigr)-\phi^\star+\frac{1}{\gamma(k)}D_\omega\bigl(x(k+1),x^\star\bigr)&\leq \frac{1}{2\eta(k)}\bigl\|e\bigl(d(k)\bigr)\bigr\|^2_*\nonumber\\
&\hspace{0.5cm}+\bigl\langle e\bigl(d(k)\bigr),x(k)-x^\star \bigr\rangle+\frac{1}{\gamma(k)}D_\omega\bigl(x(k),x^\star\bigr)\nonumber\\
&\hspace{0.5cm}+\frac{L(\tau_{\max}+1)}{2}\sum_{j=0}^{\tau_{\max}} \bigl\|x(k-j)-x(k-j+1)\bigr\|^2\nonumber\\
&\hspace{0.5cm}-\frac{1}{2}\left(\frac{1}{\gamma(k)}-\eta(k)\right)\|x(k+1)-x(k) \|^2\nonumber\\
&\hspace{0.5cm}-\frac{\mu_\Psi}{2}\bigl\|x(k+1)- x^\star)\bigr\|^2,
\label{Lemma 1-1}
\end{align}
\normalsize
where $x^{\star}\in X^{\star}$, $\{\eta(k)\}$ is a sequence of strictly positive numbers, and $e(k):=\nabla f(x(k))-{g}_{\textup{ave}}(k)$ is the error in the gradient estimate.
\end{lemma}

\begin{proof}
We start with the first-order optimality condition for the point $x(k+1)$ in the minimization problem~\eqref{Async. Iteration}: there exists subgradient $s(k+1)\in\partial \Psi(x(k+1))$ such that for all $z\in \textup{dom}\; \Psi$, we have
\begin{align*}
\biggl\langle {g}_{\textup{ave}}\bigl(d(k)\bigr)+s(k+1)+\frac{1}{\gamma(k)}\nabla_{(2)} D_\omega\bigl(x(k),x(k+1)\bigr),z-x(k+1) \biggr\rangle \geq 0,
\end{align*}
where $\nabla_{(2)} D_\omega(\cdot,\cdot)$ denotes the partial derivative of the Bregman distance function with respect to the second variable. Plugging the following equality
\begin{align*}
\nabla_{(2)} D_\omega\bigl(x(k),x(k+1)\bigr)=\nabla \omega \bigl(x(k+1)\bigr)-\nabla \omega \bigl(x(k)\bigr),
\end{align*}
into the previous inequality and re-arranging terms gives
\small
\begin{align}
\frac{1}{\gamma(k)}\biggl\langle \nabla\omega \bigl(x(k)\bigr)-\nabla \omega \bigl(x(k+1)\bigr),z-x(k+1)\biggr \rangle&\leq \biggl\langle {g}_{\textup{ave}}\bigl(d(k)\bigr)+s(k+1),z-x(k+1) \biggr\rangle\nonumber\\
&= \biggl\langle {g}_{\textup{ave}}\bigl(d(k)\bigr),z-x(k+1) \biggr\rangle\nonumber\\
&\hspace{0.5cm}+\biggl\langle s(k+1),z-x(k+1) \biggr\rangle\nonumber\\
&\leq \biggl\langle {g}_{\textup{ave}}\bigl(d(k)\bigr),z-x(k+1) \biggr\rangle\nonumber\\
&\hspace{0.5cm}+\Psi(z)-\Psi(x(k+1))-\frac{\mu_\Psi}{2}\bigl\|z-x(k+1)\bigr\|^2,
\label{Proof 1}
\end{align}
\normalsize
where the last inequality used
\begin{align*}
\Psi(z)\geq \Psi\bigl(x(k+1)\bigr)+\bigl\langle s(k+1),z-x(k+1)\bigr\rangle+\frac{\mu_\Psi}{2}\bigl\|z-x(k+1)\bigr\|^2,
\end{align*}
by the (strong) convexity of $\Psi$. We now use the following well-known \textit{three point identity} of the Bregman distance function~\cite{Chen:93} to rewrite the left-hand side of~\eqref{Proof 1}:
\begin{align*}
\bigl\langle \nabla \omega(a)-\nabla \omega(b),c-b\bigr\rangle=D_\omega(a,b)-D_\omega(a,c)+D_\omega(b,c).
\end{align*}
From this relation, with $a=x(k)$, $b=x(k+1)$, and $c=z$, we have
\small
\begin{align*}
\biggl\langle \nabla\omega \bigl(x(k)\bigr)&-\nabla \omega \bigl(x(k+1)\bigr),z-x(k+1)\biggr \rangle=D_\omega\bigl(x(k),x(k+1)\bigr)-D_\omega\bigl(x(k),z\bigr)+D_\omega\bigl(x(k+1),z\bigr).
\end{align*}
\normalsize
Substituting the preceding equality into~\eqref{Proof 1} and re-arranging terms result in
\small
\begin{align*}
\Psi(x(k+1))-\Psi(z)+\frac{1}{\gamma(k)}D_\omega\bigl(x(k+1),z\bigr)&\leq \biggl\langle {g}_{\textup{ave}}\bigl(d(k)\bigr),z-x(k+1) \biggr\rangle+\frac{1}{\gamma(k)}D_\omega\bigl(x(k),z\bigr)\\
&\hspace{0.5cm}-\frac{1}{\gamma(k)}D_\omega\bigl(x(k),x(k+1)\bigr)-\frac{\mu_\Psi}{2}\bigl\|z-x(k+1)\bigr\|^2.
\end{align*}
\normalsize
Since the distance generating function $\omega(x)$ is $1$-strongly convex, we have the lower bound
\begin{align*}
D_{\omega}\bigl(x(k),x(k+1)\bigr)\geq \frac{1}{2}\|x(k+1)-x(k) \|^2,
\end{align*}
which implies that
\small
\begin{align}
\Psi(x(k+1))-\Psi(z)+\frac{1}{\gamma(k)}D_\omega\bigl(x(k+1),z\bigr)&\leq \biggl\langle {g}_{\textup{ave}}\bigl(d(k)\bigr),z-x(k+1) \biggr\rangle+\frac{1}{\gamma(k)}D_\omega\bigl(x(k),z\bigr)\nonumber\\
&\hspace{0.5cm}-\frac{1}{2\gamma(k)}\|x(k+1) - x(k)\|^2-\frac{\mu_\Psi}{2}\bigl\|z-x(k+1)\bigr\|^2.
\label{Proof 2}
\end{align}
\normalsize

The essential idea in the rest of the proof is to use convexity and smoothness of the expectation function $f$ to bound $f(x(k+1)-f(z)$ for each $z\in \textup{dom}\; \Psi$ and each $k\in \mathbb{N}_0$.  According to Assumption~\ref{Assumption 2}, $\nabla F(x,\xi)$ and, hence, $\nabla f(x)$ are Lipschitz continuous with the constant $L$. By using the $L$-Lipschitz continuity of $\nabla f$ and then the convexity of $f$, we have
\begin{align}
f(x(k+1))&\leq f(x(d(k)))+\langle \nabla f(x(d(k))), x(k+1)-x(d(k))\rangle+\frac{L}{2}\|x(k+1)-x(d(k)) \|^2\nonumber\\
&\leq f\bigl(z\bigr)+\bigl\langle \nabla f\bigl(x(d(k))\bigr), x(k+1)-z\bigr \rangle+\frac{L}{2}\bigl\|x(k+1)-x(d(k)) \bigr\|^2,
\label{Proof 3}
\end{align}
for any $z\in \textup{dom}\; \Psi$. Combining inequalities~\eqref{Proof 2} and~\eqref{Proof 3}, and recalling that $\phi(x)=f(x)+\Psi(x)$, we obtain
\small
\begin{align*}
\phi(x(k+1))-\phi(z)+\frac{1}{\gamma(k)}D_\omega\bigl(x(k+1),z\bigr)&\leq \bigl\langle \nabla f\bigl(x(d(k))\bigr)-{g}_{\textup{ave}}\bigl(d(k)\bigr),x(k+1)-z \bigr\rangle+\frac{1}{\gamma(k)}D_\omega\bigl(x(k),z\bigr)\nonumber\\
&\hspace{0.5cm}-\frac{1}{2\gamma(k)}\|x(k+1)-x(k) \|^2-\frac{\mu_\Psi}{2}\bigl\|z-x(k+1)\bigr\|^2\nonumber\\
&\hspace{0.5cm}+\frac{L}{2}\|x(k+1)-x(d(k)) \|^2.\nonumber
\end{align*}
\normalsize
We now rewrite the above inequality in terms of the error $e(d(k))=\nabla f(x(d(k)))-{g}_{\textup{ave}}(d(k))$ as follows:
\small
\begin{align}
\phi(x(k+1))-\phi(z)+\frac{1}{\gamma(k)}D_\omega\bigl(x(k+1),z\bigr)&\leq \bigl\langle e\bigl(d(k)\bigr),x(k+1)-z \bigr\rangle+\frac{1}{\gamma(k)}D_\omega\bigl(x(k),z\bigr)\nonumber\\
&\hspace{0.5cm}-\frac{1}{2\gamma(k)}\|x(k+1)- x(k) \|^2-\frac{\mu_\Psi}{2}\bigl\|z-x(k+1)\bigr\|^2\nonumber\\
&\hspace{0.5cm}+\frac{L}{2}\|x(k+1)-x(d(k)) \|^2\nonumber\\
&=\underbrace{\bigl\langle e\bigl(d(k)\bigr),x(k+1)-x(k) \bigr\rangle}_{U_1}\nonumber\\
&\hspace{0.5cm}+\bigl\langle e\bigl(d(k)\bigr),x(k)-z \bigr\rangle+\frac{1}{\gamma(k)}D_\omega\bigl(x(k),z\bigr)\nonumber\\
&\hspace{0.5cm}-\frac{1}{2\gamma(k)}\|x(k+1) - x(k)\|^2-\frac{\mu_\Psi}{2}\bigl\|z-x(k+1)\bigr\|^2\nonumber\\
&\hspace{0.5cm}+\frac{L}{2}\underbrace{\|x(k+1)-x(d(k)) \|^2}_{U_2}.
\label{Proof 4}
\end{align}
\normalsize
We will seek upper bounds on the quantities $U_1$ and $U_2$. Let $\{\eta(k)\}_{k\in\mathbb{N}_0}$ be a sequence of positive numbers. For $U_1$, we have
\begin{align}
U_1&\leq \left|\left\langle\frac{1}{\sqrt{\eta(k)}}e\bigl(d(k)\bigr),\sqrt{\eta(k)} \bigl(x(k+1)-x(k)\bigr)\right\rangle\right|\nonumber\\
&\leq \frac{1}{2\eta(k)}\bigl\|e\bigl(d(k)\bigr)\bigr\|^2_*+\frac{{\eta(k)}}{2}\bigl\|x(k+1)-x(k) \bigr\|^2,
\label{Proof 5}
\end{align}
where the second inequality follows from the Fenchel-Young inequality applied to the conjugate pair $\frac{1}{2}\|\cdot\|^2$ and $\frac{1}{2}\|\cdot\|_*^2$, \textit{i.e.},
\begin{align*}
\bigl|\langle a,b\rangle \bigr|\leq \frac{1}{2}\bigl\|a\bigr\|_*^2+\frac{1}{2}\bigl\|b\bigr\|^2.
\end{align*}
We now turn to $U_2$. It follows from definition $\tau(k)=k-d(k)$ that
\begin{align*}
U_2&=\bigl(k-d(k)+1\bigr)^2\left\|\sum_{j=0}^{k-d(k)}\frac{x(k-j)-x(k-j+1)}{k-d(k)+1}\right\|^2\nonumber\\
&{=}\bigl(\tau(k)+1\bigr)^2\left\|\sum_{j=0}^{\tau(k)}\frac{x(k-j)-x(k-j+1)}{\tau(k)+1}\right\|^2.
\end{align*}
Then, by the convexity of the norm $\|\cdot \|$, we conclude that
\begin{align}
U_2&\leq \bigl(\tau(k)+1\bigr)\sum_{j=0}^{\tau(k)} \bigl\|x(k-j)-x(k-j+1)\bigr\|^2\nonumber\\
&\leq \bigl(\tau_{\max}+1\bigr)\sum_{j=0}^{\tau_{\max}} \bigl\|x(k-j)-x(k-j+1)\bigr\|^2,
\label{Proof 6}
\end{align}
where the last inequality comes from our assumption that $\tau(k)\leq \tau_{\max}$ for all $k\in \mathbb{N}_0$. Substituting inequalities~\eqref{Proof 5} and~\eqref{Proof 6} into the bound~\eqref{Proof 4} and simplifying yield
\small
\begin{align*}
\phi(x(k+1))-\phi(z)+\frac{1}{\gamma(k)}D_\omega\bigl(x(k+1),z\bigr)&\leq \frac{1}{2\eta(k)}\bigl\|e\bigl(d(k)\bigr)\bigr\|^2_*\nonumber\\
&\hspace{0.5cm}+\bigl\langle e\bigl(d(k)\bigr),x(k)-z \bigr\rangle+\frac{1}{\gamma(k)}D_\omega\bigl(x(k),z\bigr)\nonumber\\
&\hspace{0.5cm}+\frac{L(\tau_{\max}+1)}{2}\sum_{j=0}^{\tau_{\max}} \bigl\|x(k-j)-x(k-j+1)\bigr\|^2\nonumber\\
&\hspace{0.5cm}-\frac{1}{2}\left(\frac{1}{\gamma(k)}-\eta(k)\right)\|x(k+1)- x(k) \|^2\nonumber\\
&\hspace{0.5cm}-\frac{\mu_\Psi}{2}\bigl\|z-x(k+1)\bigr\|^2.
\end{align*}
\normalsize
Setting $z=x^\star$, where $x^{\star}\in X^{\star}$, completes the proof.
\end{proof}

The next result follows from Lemma~\ref{Lemma 1} by taking summation of the relations in~\eqref{Lemma 1-1}.

\begin{lemma}
\label{Lemma 2}
Let Assumptions~\ref{Assumption 1}--\ref{Assumption 4} hold. Assume also that $\{\gamma(k)\}_{k\in\mathbb{N}_0}$ is set to
\begin{align*}
\gamma(k)=\frac{1}{\eta(k)+L(\tau_{\max}+1)^2},\quad k\in\mathbb{N}_0,
\end{align*}
where $\eta(k)$ is positive for all $k$. Then, the iterates $\{x(k)\}_{k\in\mathbb{N}_0}$ produced by Algorithm~\ref{Algorithm 1} satisfy
\small
\begin{align*}
\sum_{k=0}^{T-1}\bigl(\phi(x(k+1))-\phi^\star\bigr)&\leq \sum_{k=0}^{T-1}\frac{1}{2\eta(k)}\bigl\|e\bigl(d(k)\bigr)\bigr\|^2_*\nonumber\\
&\hspace{0.5cm}+\sum_{k=0}^{T-1}\bigl\langle e\bigl(d(k)\bigr),x(k)-x^\star \bigr\rangle+ \frac{1}{\gamma(0)}D_\omega\bigl(x(0),x^\star\bigr)\nonumber\\
&\hspace{0.5cm}+\sum_{k=0}^{T-1}\left(\frac{1}{\gamma(k+1)}-\frac{1}{\gamma(k)}\right)D_\omega\bigl(x(k+1),x^\star\bigr)\\
&\hspace{0.5cm}-\frac{\mu_\Psi}{2}\sum_{k=0}^{T-1}\bigl\|x(k+1)-x^\star\bigr\|^2,
\end{align*}
\normalsize
for all $T\in\mathbb{N}$.
\end{lemma}

\begin{proof}
Applying Lemma~\ref{Lemma 1} with
\begin{align*}
\eta(k)=\frac{1}{\gamma(k)}-L(\tau_{\max}+1)^2,
\end{align*}
adding and subtracting $\gamma(k+1)^{-1}D_\omega\bigl(x(k+1),x^\star\bigr)$ to the left-hand side of~\eqref{Lemma 1-1}, and re-arranging terms, we obtain
\small
\begin{align*}
\phi(x(k+1))-\phi^\star+\frac{1}{\gamma(k+1)}D_\omega\bigl(x(k+1),x^\star\bigr)&\leq \frac{1}{2\eta(k)}\bigl\|e\bigl(d(k)\bigr)\bigr\|^2_*\nonumber\\
&\hspace{0.5cm}+\bigl\langle e\bigl(d(k)\bigr),x(k)-x^\star \bigr\rangle + \frac{1}{\gamma(k)}D_\omega\bigl(x(k),x^\star\bigr)\nonumber\\
&\hspace{0.5cm}+\left(\frac{1}{\gamma(k+1)}-\frac{1}{\gamma(k)}\right)D_\omega\bigl(x(k+1),x^\star\bigr)\\
&\hspace{0.5cm}+\frac{L(\tau_{\max}+1)}{2}\sum_{j=0}^{\tau_{\max}} \bigl\|x(k-j)-x(k-j+1)\bigr\|^2\nonumber\\
&\hspace{0.5cm}-\frac{L(\tau_{\max}+1)^2}{2}\|x(k+1) - x(k)\|^2\\
&\hspace{0.5cm}-\frac{\mu_\Psi}{2}\bigl\|x(k+1)-x^\star\bigr\|^2.
\end{align*}
\normalsize
Summing the preceding inequality over $k=0,\ldots,T-1$, $T\in\mathbb{N}$, yields
\small
\begin{align}
\sum_{k=0}^{T-1}\bigl(\phi(x(k+1))-\phi^\star\bigr)+\frac{1}{\gamma(T)}D_\omega\bigl(x(T),x^\star\bigr)&\leq \sum_{k=0}^{T-1}\frac{1}{2\eta(k)}\bigl\|e\bigl(d(k)\bigr)\bigr\|^2_*\nonumber\\
&\hspace{0.5cm}+\sum_{k=0}^{T-1}\bigl\langle e\bigl(d(k)\bigr),x(k)-x^\star \bigr\rangle+\frac{1}{\gamma(0)}D_\omega\bigl(x(0),x^\star\bigr)\nonumber\\
&\hspace{0.5cm}+\sum_{k=0}^{T-1}\left(\frac{1}{\gamma(k+1)}-\frac{1}{\gamma(k)}\right)D_\omega\bigl(x(k+1),x^\star\bigr)\nonumber\\
&\hspace{0.5cm}+\frac{L(\tau_{\max}+1)}{2}\sum_{k=0}^{T-1}\sum_{j=0}^{\tau_{\max}} \bigl\|x(k-j)-x(k-j+1)\bigr\|^2\nonumber\\
&\hspace{0.5cm}-\frac{L(\tau_{\max}+1)^2}{2}\sum_{k=0}^{T-1}\|x(k+1)-x(k) \|^2\nonumber\\
&\hspace{0.5cm}-\frac{\mu_\Psi}{2}\sum_{k=0}^{T-1}\bigl\|x(k+1)-x^\star\bigr\|^2\nonumber\\
&\hspace{0.0cm}\leq \sum_{k=0}^{T-1}\frac{1}{2\eta(k)}\bigl\|e\bigl(d(k)\bigr)\bigr\|^2_*\nonumber\\
&\hspace{0.5cm}+\sum_{k=0}^{T-1}\bigl\langle e\bigl(d(k)\bigr),x(k)-x^\star \bigr\rangle+\frac{1}{\gamma(0)}D_\omega\bigl(x(0),x^\star\bigr)\nonumber\\
&\hspace{0.5cm}+\sum_{k=0}^{T-1}\left(\frac{1}{\gamma(k+1)}-\frac{1}{\gamma(k)}\right)D_\omega\bigl(x(k+1),x^\star\bigr)\nonumber\\
&\hspace{0.5cm}-\frac{\mu_\Psi}{2}\sum_{k=0}^{T-1}\bigl\|x(k+1)-x^\star\bigr\|^2,
\label{Proof 7}
\end{align}
\normalsize
where the second inequality used the facts
\begin{align*}
\sum_{k=0}^{T-1}\sum_{j=0}^{\tau_{\max}} \bigl\|x(k-j)-x(k-j+1)\bigr\|^2&=\sum_{j=0}^{\tau_{\max}}\sum_{k=-j}^{T-j-1} \bigl\|x(k)-x(k+1)\bigr\|^2\\
&=\sum_{j=0}^{\tau_{\max}}\sum_{k=0}^{T-j-1} \bigl\|x(k)-x(k+1)\bigr\|^2\\
&\leq \sum_{j=0}^{\tau_{\max}}\sum_{k=0}^{T-1} \bigl\|x(k)-x(k+1)\bigr\|^2\\
&\leq (\tau_{\max}+1)\sum_{k=0}^{T-1} \bigl\|x(k)-x(k+1)\bigr\|^2,
\end{align*}
and $x(k)=x(0)$ for all $k\leq 0$. Dropping the second term on the left-hand side of~\eqref{Proof 7} concludes the proof.
\end{proof}

\begin{lemma}
\label{Lemma 3}
Let $\|\cdot \|$ be a norm over $\mathbb{R}^n$ and let $\|\cdot \|_{\star}$ be its dual norm. Let $\omega$ be a $1$-strongly convex function with respect to $\|\cdot \|$ over $\mathbb{R}^n$. If $y_1,\ldots,y_b\in\mathbb{R}^n$ are mean zero random variables drawn i.i.d. from a distribution $\mathcal{P}$, then
\begin{align*}
\mathbb{E}\left[\left\|\frac{1}{b}\sum_{i=1}^by_i\right\|^2_*\right]\leq \frac{c}{b^2}\sum_{i=1}^b\mathbb{E}\left[\left\|y_i\right\|^2_*\right],
\end{align*}
where $c\in[1,b]$ is given by
\begin{align*}
 c= \left\{
  \begin{array}{ll}
    1, & \textup{if}\;\|\cdot\|_*=\|\cdot\|_2,\\
    {2\max_{\|x\|=1}\omega(x)},&\textup{otherwise}.
  \end{array}
   \right.
\end{align*}
\end{lemma}

\begin{proof}
The result follows from~\cite[Lemma B.2]{Cotter:11} and convexity of the norm $\|\cdot\|_*$. For further details, see~\cite[\S 4.1]{Dekel:12}.
\end{proof}

%
%

\subsection{Proof of Theorem~\ref{Theorem 1}}
\label{Theorem:1}

Assume that the step-size $\{\gamma(k)\}_{k\in\mathbb{N}_0}$ is set to
\begin{align*}
\gamma(k)=\gamma=\frac{1}{\eta+L(\tau_{\max}+1)^2},
\end{align*}
for some $\eta>0$. It is clear that $\gamma$ satisfies~\eqref{Fixed Step-size}. Applying Lemma~\ref{Lemma 2} with $\mu_\Psi=0$, $\gamma(k)=\gamma$ and $\eta(k)=\eta$, we obtain
\small
\begin{align}
\sum_{k=0}^{T-1}\bigl(\phi(x(k+1))-\phi^\star\bigr)&\leq \sum_{k=0}^{T-1}\frac{1}{2\eta}\bigl\|e\bigl(d(k)\bigr)\bigr\|^2_*+\sum_{k=0}^{T-1}\bigl\langle e\bigl(d(k)\bigr),x(k)-x^\star \bigr\rangle+ \frac{D_\omega\bigl(x(0),x^{\star}\bigr)}{\gamma},
\label{Proof 8}
\end{align}
\normalsize
for all $T\in\mathbb{N}$. Each $x(k)$, $k\in\mathbb{N}$, is a deterministic function of the history $\xi_{[k-1]}:=\{\xi_i(t)\;|\;i=1,\ldots,b,\;t=0,\ldots,k-1\}$ but not of $\xi_i(k)$. Since $\nabla f(x)=\mathbb{E}_\xi [\nabla_x F(x,\xi)]$, it follows that
\begin{align*}
\mathbb{E}_{|\xi_{[k-1]}}\left[\bigl\langle e\bigl(d(k)\bigr),x(k)-x^\star \bigr\rangle\right]=0.
\end{align*}
Moreover, as $\xi_i$ and $\xi_j$ are independent whenever $i\neq j$, it follows from Lemma~\ref{Lemma 3} that
\begin{align*}
\mathbb{E}\bigl[\|e\bigl(d(k)\bigr)\|^2_*\bigr]&=\mathbb{E}\left[\left\|\frac{1}{b}\sum_{i=1}^b\bigl(\nabla f(x(d(k)))-\nabla_x F(x(d(k)),\xi_i)\bigr)\right\|^2_*\right]\\
&\leq \frac{c}{b^2}\sum_{i=1}^b\mathbb{E}\left[\left\|\nabla f(x(d(k)))-\nabla_x F(x(d(k)),\xi_i)\right\|^2_*\right]\\
&\leq \frac{c\sigma^2}{b},
\end{align*}
where the last inequality follows from Assumption~\ref{Assumption 3}. Taking expectation on both sides of~\eqref{Proof 8} and using the above observations yield
\begin{align*}
\sum_{k=1}^T\bigl(\mathbb{E}[\phi(x(k))]-\phi^{\star}\bigr)\leq \frac{c\sigma^2}{2\eta b}T+\frac{D_\omega\bigl(x(0),x^{\star}\bigr)}{\gamma}.
\end{align*}
By the convexity of $\phi$, we have
\begin{align*}
\phi\bigl({x}_{\textup{ave}}(T)\bigr)=\phi\left(\frac{1}{T}\sum_{k=1}^T x(k)\right)\leq \frac{1}{T}\sum_{k=1}^T \phi\bigl(x(k)\bigr),
\end{align*}
which implies that
\begin{align*}
\mathbb{E}\bigl[\phi({x}_{\textup{ave}}\bigl(T)\bigr)\bigr]-\phi^{\star}\leq \frac{c\sigma^2}{2\eta b}+\frac{D_\omega\bigl(x(0),x^{\star}\bigr)}{\gamma T}.
\end{align*}
Substituting $\eta=\gamma^{-1}-L(\tau_{\max}+1)^2$ into the above inequality proves the theorem.

%
%

\subsection{Proof of Theorem~\ref{Theorem 2}}
\label{Theorem:2}

Assume that the step-size $\{\gamma(k)\}_{k\in\mathbb{N}_0}$ is chosen such that $\gamma(k)^{-1}=L(\tau_{\max}+1)^2+\alpha(k)$ where
\begin{align*}
\alpha(k)=\frac{\sigma\sqrt{c}\sqrt{k+1}}{R\sqrt{b}}.
\end{align*}
Since $\gamma(k)$ is a non-increasing sequence, and $D_\omega(x,y)\leq R^2$ for all $x,y\in \;\textup{dom}\;\Psi$, we have
\begin{align*}
\sum_{k=0}^{T-1}\left(\frac{1}{\gamma(k+1)}-\frac{1}{\gamma(k)}\right)&D_\omega\bigl(x(k+1),x^{\star}\bigr)\leq \left(\frac{1}{\gamma(T)}-\frac{1}{\gamma(0)}\right)R^2.
\end{align*}
Applying Lemma~\ref{Lemma 2} with $\mu_\Psi=0$ and $\eta(k)=\alpha(k)$, taking expecation, and using Lemma~\ref{Lemma 3} completely identically to the proof of Theorem~\ref{Theorem 1}, we then obtain
\begin{eqnarray}
\sum_{k=1}^T\bigl(\mathbb{E}[\phi(x(k))]-\phi^{\star}\bigr)\leq \frac{R^2}{\gamma(T)}+\frac{c\sigma^2}{2 b}\sum_{k=0}^{T-1}\frac{1}{\alpha(k)}.
\label{Proof Theorem 3-1}
\end{eqnarray}
Viewing the sum as an lower-estimate of the integral of the function $y(t)=1/\sqrt{t+1}$, one can verify that
\begin{align*}
\sum_{k=0}^{T-1}\frac{1}{\alpha(k)}=\sum_{k=0}^{T-1} \frac{1}{\widetilde{\alpha}\sqrt{k+1}}&\leq\frac{1}{{\widetilde{\alpha}}}\left(1+\int_0^{T-1}\frac{dt}{\sqrt{t+1}}\right)\\
&\leq \frac{2 \sqrt{T}}{\widetilde{\alpha}},
\end{align*}
where $\widetilde{\alpha}=(\sigma\sqrt{c})/(R \sqrt{b})$. Substituting this inequality into the bound~\eqref{Proof Theorem 3-1}, we obtain the claimed guaranteed bound.

%
%

\subsection{Proof of Theorem~\ref{Theorem 3}}
\label{Theorem:3}

Assume that the step-size $\{\gamma(k)\}_{k\in\mathbb{N}_0}$ in Algorithm~\ref{Algorithm 1} is set to $\gamma(k)^{-1}=2L(\tau_{\max}+1)^2+\beta(k)$, with
\begin{align*}
\beta(k)=\frac{\mu_\Psi}{3Q}\bigl(k+\tau_{\max}+1\bigr).
\end{align*}
We first describe some important properties of $\gamma(k)$ relevant to our proof. Clearly, $\gamma(k)$ is non-increasing, \textit{i.e.},
\begin{align}
\frac{1}{\gamma(k)}&\leq\frac{1}{\gamma(k+1)},\label{Proof Theorem 3:0:0}
\end{align}
for all $k\in\mathbb{N}_0$. Since $\gamma(0)^{-1}\leq \gamma(k)^{-1}$, we have
\begin{align}
2L(\tau_{\max}+1)^2+\frac{\mu_\Psi\tau_{\max}}{3Q}&\leq \frac{1}{\gamma(k)}.
\label{Proof Theorem 3:0}
\end{align}
Moreover, one can easily verify that
\begin{align*}
\frac{1}{\gamma(k+1)^{2}}-\frac{1}{\gamma(k)^{2}}&=\frac{\mu_\Psi}{Q}\left(\frac{4L}{3}(\tau_{\max}+1)^2+\frac{\mu_\Psi}{3Q}\left(\frac{2}{3}(k+\tau_{\max})+1\right)\right)\\
&\leq\frac{\mu_\Psi}{Q}\left(2L(\tau_{\max}+1)^2+\frac{\mu_\Psi}{3Q}\bigl(k+\tau_{\max}+1\bigr)\right)\\
&=\frac{\mu_\Psi}{Q}\frac{1}{\gamma(k)},
\end{align*}
which implies that
\begin{align}
\frac{1}{\gamma(k+1)^{2}}\leq\frac{1}{\gamma(k)}\left(\frac{1}{\gamma(k)}+\frac{\mu_\Psi}{Q}\right),
\label{Proof Theorem 3:1}
\end{align}
for all $k\in\mathbb{N}_0$. Finally, by the definition of $\gamma(k)$, we have
\begin{align*}
\frac{\gamma(k)}{\gamma(k+\tau_{\max})}&=1+\frac{\frac{\mu_\Psi}{3Q}\tau_{\max}}{2L(\tau_{\max}+1)^2+\frac{\mu_\Psi}{3Q}\bigl(k+\tau_{\max}+1\bigr)}\\
&\leq 1+\frac{\mu_\Psi\tau_{\max}}{6LQ(\tau_{\max}+1)^2},
\end{align*}
and hence,
\begin{align}
\frac{1}{\gamma(k+\tau_{\max})}\leq \left(1+\frac{\mu_\Psi\tau_{\max}}{6LQ(\tau_{\max}+1)^2}\right)\frac{1}{\gamma(k)}.
\label{Proof Theorem 3:2}
\end{align}

We are now ready to prove Theorem~\ref{Theorem 3}. Applying Lemma~\ref{Lemma 1} with
\begin{align*}
\eta(k)=\frac{1}{2\gamma(k)},\quad k\in\mathbb{N}_0,
\end{align*}
and using the fact
\begin{align*}
D_{\omega}\bigl(x(k+1),x^\star\bigr)\leq \frac{Q}{2}\bigl\| x(k+1)- x^\star\bigr\|^2,
\end{align*}
by Assumption~\ref{Assumption 5}, we obtain
\small
\begin{align*}
\phi\bigl(x(k+1)\bigr)-\phi^\star+\left(\frac{1}{\gamma(k)}+\frac{\mu_\Psi}{Q}\right)D_\omega\bigl(x(k+1),x^\star\bigr)&\leq \gamma(k)\bigl\|e\bigl(d(k)\bigr)\bigr\|^2_*\nonumber\\
&\hspace{0.5cm}+\bigl\langle e\bigl(d(k)\bigr),x(k)-x^\star \bigr\rangle+\frac{1}{\gamma(k)}D_\omega\bigl(x(k),x^\star\bigr)\nonumber\\
&\hspace{0.5cm}+\frac{L(\tau_{\max}+1)}{2}\sum_{j=0}^{\tau_{\max}} \bigl\|x(k-j)-x(k-j+1)\bigr\|^2\nonumber\\
&\hspace{0.5cm}-\frac{1}{4\gamma(k)}\|x(k+1) -x(k)\|^2\nonumber.
\end{align*}
\normalsize
Multiplying both sides of this relation by $1/\gamma(k)$, and then using \eqref{Proof Theorem 3:1}, we have
\small
\begin{align*}
\frac{1}{\gamma(k)}\bigl(\phi\bigl(x(k+1)\bigr)-\phi^\star\bigr)+\frac{1}{\gamma(k+1)^2}D_\omega\bigl(x(k+1),x^\star\bigr)&\leq \bigl\|e\bigl(d(k)\bigr)\bigr\|^2_*\nonumber\\
&\hspace{0.5cm}+\frac{1}{\gamma(k)}\bigl\langle e\bigl(d(k)\bigr),x(k)-x^\star \bigr\rangle+\frac{1}{\gamma(k)^2}D_\omega\bigl(x(k),x^\star\bigr)\nonumber\\
&\hspace{0.5cm}+\frac{L(\tau_{\max}+1)}{2\gamma(k)}\sum_{j=0}^{\tau_{\max}} \bigl\|x(k-j)-x(k-j+1)\bigr\|^2\nonumber\\
&\hspace{0.5cm}-\frac{1}{4\gamma(k)^2}\|x(k+1) - x(k)\|^2\nonumber.
\end{align*}
\normalsize
Summing the above inequality from $k = 0$ to $k = T-1$, $T\in\mathbb{N}$, and dropping the first term on the left-hand side yield
\small
\begin{align}
\frac{1}{\gamma(T)^2}D_\omega\bigl(x(T),x^\star\bigr)&\leq \sum_{k=0}^{T-1}\bigl\|e\bigl(d(k)\bigr)\bigr\|^2_*\nonumber\\
&\hspace{0.5cm}+\sum_{k=0}^{T-1}\frac{1}{\gamma(k)}\bigl\langle e\bigl(d(k)\bigr),x(k)-x^\star \bigr\rangle+\frac{1}{\gamma(0)^2}D_\omega\bigl(x(0),x^\star\bigr)\nonumber\\
&\hspace{0.5cm}+\frac{L(\tau_{\max}+1)}{2}\sum_{k=0}^{T-1}\sum_{j=0}^{\tau_{\max}}\frac{1}{\gamma(k)} \bigl\|x(k-j)-x(k-j+1)\bigr\|^2\nonumber\\
&\hspace{0.5cm}-\frac{1}{4}\sum_{k=0}^{T-1}\frac{1}{\gamma(k)^2}\|x(k+1) - x(k)\|^2.
\label{Proof Theorem 3:3}
\end{align}
\normalsize
What remains is to bound the third term on the right-hand side of~\eqref{Proof Theorem 3:3}. It follows from~\eqref{Proof Theorem 3:0:0}--\eqref{Proof Theorem 3:2} that

\small
\begin{align*}
\frac{L(\tau_{\max}+1)}{2}\sum_{k=0}^{T-1}\sum_{j=0}^{\tau_{\max}}\frac{1}{\gamma(k)} \bigl\|x(k-j)-x(k-j+1)\bigr\|^2&=\frac{L(\tau_{\max}+1)}{2}\sum_{j=0}^{\tau_{\max}}\sum_{k=0}^{T-j-1}\frac{1}{\gamma(k+j)} \bigl\|x(k)-x(k+1)\bigr\|^2\\
&\leq \frac{L(\tau_{\max}+1)}{2}\sum_{j=0}^{\tau_{\max}}\sum_{k=0}^{T-1} \frac{1}{\gamma(k+j)}\bigl\|x(k)-x(k+1)\bigr\|^2\\
&\overset{\eqref{Proof Theorem 3:0:0}}{\leq} \frac{L(\tau_{\max}+1)}{2}\sum_{j=0}^{\tau_{\max}}\sum_{k=0}^{T-1}\frac{1}{\gamma(k+\tau_{\max})} \bigl\|x(k)-x(k+1)\bigr\|^2\\
&=\frac{L(\tau_{\max}+1)^2}{2}\sum_{k=0}^{T-1}\frac{1}{\gamma(k+\tau_{\max})} \bigl\|x(k)-x(k+1)\bigr\|^2\\
&\overset{\eqref{Proof Theorem 3:2}}{\leq}\frac{2L(\tau_{\max}+1)^2+\frac{\mu_\Psi\tau_{\max}}{3Q}}{4}\sum_{k=0}^{T-1}\frac{1}{\gamma(k)} \bigl\|x(k)-x(k+1)\bigr\|^2\\
&\overset{\eqref{Proof Theorem 3:0}}{\leq}\frac{1}{4}\sum_{k=0}^{T-1}\frac{1}{\gamma(k)^2} \bigl\|x(k)-x(k+1)\bigr\|^2.
\end{align*}
\normalsize
Substituting the above inequality into~\eqref{Proof Theorem 3:3}, and then taking expectation on both sides (similarly to the proof of Theorems~\ref{Theorem 1} and~\ref{Theorem 2}), we have
\begin{align}
&\frac{1}{\gamma(T)^2}\mathbb{E}\bigl[D_\omega\bigl(x(T),x^\star\bigr)\bigr]\leq \frac{c\sigma^2 T}{b}+\frac{1}{\gamma(0)^2}D_\omega\bigl(x(0),x^\star\bigr).
\label{Proof Theorem 3:4}
\end{align}
According to Remark~\ref{Remark 1},
\begin{align*}
\frac{1}{2}\| x(T)- x^\star\|^2\leq D_\omega\bigl(x(T),x^\star\bigr).
\end{align*}
Moreover, by the definition of $\gamma(k)$,
\begin{align*}
 \frac{\mu_\Psi (T+1)}{3Q} \leq \beta(T)\leq \frac{1}{\gamma(T)}.
\end{align*}
Combing these inequalities with the bound~\eqref{Proof Theorem 3:4}, we conclude
\begin{align*}
\mathbb{E}\bigl[\| x(T)- x^\star\|^2\bigr]\leq \frac{18 c\sigma^2 Q^2}{b\mu^2_\Psi (T+1)}+\frac{2\left(\frac{6LQ}{\mu_\Psi}+1\right)^2(\tau_{\max}+1)^4}{(T+1)^2}D_\omega\bigl(x(0),x^\star\bigr).
\end{align*}
The proof is complete.

%
%
%

\bibliographystyle{IEEEtran}
\bibliography{bibliografia}

\end{document}